\documentclass[11pt,oneside,english]{amsart}
\usepackage[T1]{fontenc}
\usepackage[latin9]{inputenc}
\usepackage{geometry}
\usepackage{color}
\usepackage{babel}
\usepackage{verbatim}
\usepackage{amsbsy}
\usepackage{amstext}
\usepackage{amsthm}
\usepackage{amssymb}
\usepackage[numbers]{natbib}
\usepackage[unicode=true,pdfusetitle,
 bookmarks=true,bookmarksnumbered=false,bookmarksopen=false,
 breaklinks=false,pdfborder={0 0 1},backref=false,colorlinks=false]
 {hyperref}
\usepackage{breakurl}

\makeatletter
\numberwithin{equation}{section}
\numberwithin{figure}{section}
  \theoremstyle{definition}
  \newtheorem{defn}{\protect\definitionname}
  \theoremstyle{plain}
  \newtheorem{lem}{\protect\lemmaname}
  \theoremstyle{remark}
  \newtheorem{rem}{\protect\remarkname}
  \theoremstyle{plain}
  \newtheorem{assumption}{\protect\assumptionname}
\theoremstyle{plain}
\newtheorem{thm}{\protect\theoremname}

\makeatother

  \providecommand{\assumptionname}{Assumption}
  \providecommand{\definitionname}{Definition}
  \providecommand{\lemmaname}{Lemma}
  \providecommand{\remarkname}{Remark}
\providecommand{\theoremname}{Theorem}

\begin{document}

\title{The Distance Function from the Boundary \linebreak{}
of a Domain with Corners}

\author{Mohammad Safdari}
\begin{abstract}
We study the regularity of the distance function to the boundary of
a domain in $\mathbb{R}^{2}$, with respect to some asymmetric norms.
We allow the boundary of the domain to have corners. We obtain an
explicit formula for the second derivative of these distance functions.
Furthermore, we study a generalized notion of the ridge of a domain,
which is the set of singularities of a distance function to the boundary
of the domain. We completely characterize the ridge by using a generalized
notion of curvature.\thanks{School of Mathematics, Institute for Research in Fundamental Sciences
(IPM), P.O. Box: 19395-5746, Tehran, Iran\protect \\
Email address: safdari@sharif.ir}
\end{abstract}

\maketitle

\section{Introduction}

The distance function from the boundary of a domain appears in different
parts of mathematical analysis, especially in the study of partial
differential equations. \citet{MR2336304} studied the distance function,
and used it to analyze PDEs of Monge\textendash Kantorovich type arising
in optimal transport theory. \citet{MR2094267} studied the distance
function in the framework of Hamilton-Jacobi equations and Finsler
geometry. \citet{itoh2001lipschitz}, and \citet{mantegazza2003hamilton}
studied the distance function in Riemannian manifolds; and \citet{CHRUSCIEL20021}
considered the case of Lorentzian space-times. On the other hand,
\citet{clarke1995proximal} and \citet{poliquin2000local} studied
the distance function in the context of nonsmooth analysis in Hilbert
spaces.

In {[}\citealp{Safdari20151}\nocite{safdari2017shape,MR1}\textendash \citealp{Safdari[18]}{]},
we examined the distance function, and used it to study variational
problems with gradient constraint, and their corresponding free boundary.
Let us introduce this problem in more detail, and see how we utilized
the distance function in its study. Let $K$ be a compact convex subset
of $\mathbb{R}^{n}$ whose interior contains the origin. We recall
from convex analysis (see \citep{MR3155183}) that the \textbf{gauge}
function of $K$ is the convex function 
\begin{equation}
\gamma_{K}(x):=\inf\{\lambda>0:x\in\lambda K\}.\label{eq: gaug}
\end{equation}
The gauge function $\gamma_{K}$ is subadditive and positively 1-homogenous,
so it looks like a norm on $\mathbb{R}^{n}$, except that $\gamma_{K}(-x)$
is not necessarily the same as $\gamma_{K}(x)$. Another notion is
that of the \textbf{polar} of $K$ 
\begin{equation}
K^{\circ}:=\{x:\langle x,y\rangle\leq1\,\textrm{ for all }y\in K\},\label{eq: K0}
\end{equation}
where $\langle\,,\rangle$ is the standard inner product on $\mathbb{R}^{n}$.
$K^{\circ}$, too, is a compact convex set containing the origin as
an interior point.

Let $U\subset\mathbb{R}^{n}$ be a bounded open set. Let $u$ be the
minimizer of 
\[
J[v]:=\int_{U}F(Dv)+g(v)\,dx,
\]
over 
\[
W_{K^{\circ}}:=\{v\in H_{0}^{1}(U):Dv\in K^{\circ}\textrm{ a.e.}\}.
\]
Note that $Dv\in K^{\circ}$ is equivalent to $\gamma_{K^{\circ}}(Dv)\le1$.
Thus $\gamma_{K^{\circ}}$ defines the gradient constraint. It can
be shown that $u$ is also the unique minimizer of $J$ over 
\[
W_{d_{K}}:=\{v\in H_{0}^{1}(U):-\bar{d}_{K}\le v\leq d_{K}\textrm{ a.e.}\},
\]
where 
\begin{align}
 & d_{K}(x)=d_{K}(x,\partial U):=\underset{y\in\partial U}{\min}\,\gamma_{K}(x-y),\nonumber \\
 & \bar{d}_{K}(x)=\bar{d}_{K}(x,\partial U):=\underset{y\in\partial U}{\min}\,\gamma_{K}(y-x).\label{eq: rho}
\end{align}
Note that $\bar{d}_{K}=d_{-K}$, since $\gamma_{-K}(\cdot)=\gamma_{K}(-\,\cdot)$.
Thus the distance function from $\partial U$ with respect to some
gauge function naturally appears in the study of variational problems
with gradient constraint.%

In \citep{Safdari20151,safdari2017shape} we investigated the above
problem in two dimensions, when $K$ is given by the $p$-norm, i.e.
when $K$ is equal to 
\[
\{(x_{1},x_{2}):(|x_{1}|^{p}+|x_{2}|^{p})^{\frac{1}{p}}\le1\}.
\]
In this regard, we had to understand the behavior of $d_{K}$. In
\citep{Safdari[18]} we generalized the above work to arbitrary dimensions,
for arbitrary gradient constraints. Here, we approximated an arbitrary
convex set $K$ with convex sets $K_{i}$ whose boundaries are $C^{2}$,
and have positive principal curvatures. This analysis relied heavily
on the properties of $d_{K}$ and $d_{K_{i}}$. Especially, we were
able to find an explicit formula for $D^{2}d_{K_{i}}$, which was
very crucial in our analysis. To the best of author's knowledge, formulas
of this kind have not appeared in the literature before, except for
the simple case of the Euclidean distance to the boundary.

In this work, we examine the regularity of $d_{K}$ in two dimensions.
We obtain formula (\ref{eq: laplace d_K}) for $D^{2}d_{K}$, when
$\partial K$ is $C^{2}$. In contrast to \citep{Safdari[18]}, here
we allow the curvature of $\partial K$ to vanish at finitely many
points. Hence this work includes the case of $p$-norms as a special
case. Also, here we allow $\partial U$ to have corners; which makes
the behavior of $d_{K}$ more complicated. In addition, we will completely
characterize the set of singularities of $d_{K}$, which we call the
$d_{K}$-ridge. 

The paper is organized as follows. In Section \ref{sec: Notation-and-Prelim}
we introduce our notation, and review some preliminary facts about
$d_{K}$ and $\gamma_{K}$. In Section \ref{sec: Regularity d} we
will study the regularity of $d_{K}$. Theorems \ref{thm: d_K is C^2},
\ref{thm: d_K is C^2, corners} are the main results of this section;
and provide detailed information about $d_{K}$. Here, we also introduce
a generalized notion of curvature, which helps us to understand the
behavior of $d_{K}$. Finally, in Section \ref{sec: Characterizing ridge}
we will characterize the set of singularities of $d_{K}$.

\section{\label{sec: Notation-and-Prelim}Notation and Preliminaries}

Let us fix some notation first. We denote by $C^{\omega}$ the space
of analytic functions (or submanifolds); so in the following when
we talk about $C^{k,\alpha}$ regularity with $k$ greater than some
fixed integer, we are also including $C^{\infty}$ and $C^{\omega}$.
We will use the abbreviations 
\begin{eqnarray*}
\gamma:=\gamma_{K}, &  & \gamma^{\circ}:=\gamma_{K^{\circ}}.
\end{eqnarray*}
We also denote the closed line segment between two points $x,y$ by
$[x,y]$, the open line segment by $]x,y[$, and the half-closed line
segments by $]x,y]$, $[x,y[$. When $d_{K}(x)=\gamma(x-y)$ for some
$y\in\partial U$, we call $y$ a \textbf{$\boldsymbol{d_{K}}$-closest}
point to $x$ on $\partial U$. Note that when $y$ is a $d_{K}$-closest
point on $\partial U$ to $x\in U$, the segment $[x,y[$ is in $U$.
We also use $B_{r}(x)$ to denote the open ball of radius $r$ centered
at $x$.

For $X\subset\mathbb{R}^{n}$, $v\in\mathbb{R}^{n}$, and $r\in\mathbb{R}$
we use the conventions 
\begin{eqnarray*}
 &  & rX:=\{rx\,\mid\,x\in X\},\\
 &  & v+X:=\{v+x\,\mid\,x\in X\}.
\end{eqnarray*}
When $n=2$, we use the notation $v^{\perp}:=(-v_{2},v_{1})$ for
the $90^{\circ}$ counterclockwise rotation of a vector $v=(v_{1},v_{2})$.
We also set $D^{\perp}f:=(Df)^{\perp}$ for a function $f$ of two
variables.

\subsection{The ridge}

First, we generalize the notion of ridge introduced by \citet{MR0184503},
and \citet{MR534111}.
\begin{defn}
The \textbf{$\boldsymbol{d_{K}}$-ridge} of $U$ is the set of all
points $x\in U$ where $d_{K}(x)=d_{K}(x,\partial U)$ is not $C^{1,1}$
in any neighborhood of $x$. We denote it by 
\[
R_{K}.
\]
\end{defn}
Recall that $K$ is a compact convex subset of $\mathbb{R}^{n}$ with
$0$ in its interior, and its gauge function $\gamma$ satisfies 
\begin{eqnarray*}
 &  & \gamma(rx)=r\gamma(x),\\
 &  & \gamma(x+y)\le\gamma(x)+\gamma(y),
\end{eqnarray*}
for all $x,y\in\mathbb{R}^{n}$ and $r\ge0$. Note that as $K$ is
closed, $K=\{\gamma\le1\}$; and as it has nonempty interior, $\partial K=\{\gamma=1\}$.
Thus, $\gamma(x-y)\le r$ is equivalent to $y\in x-rK$. Also, note
that as $B_{c}(0)\subseteq K\subseteq B_{C}(0)$ for some $C\ge c>0$,
we have 
\[
\frac{1}{C}|x|\le\gamma(x)\le\frac{1}{c}|x|,
\]
for all $x\in\mathbb{R}^{n}$. Moreover, from the definition of $d_{K}$
we easily obtain 
\begin{equation}
-\gamma(x-y)\le d_{K}(y)-d_{K}(x)\le\gamma(y-x).\label{eq: d_K Lip}
\end{equation}
Thus in particular, $d_{K}$ is Lipschitz continuous.

It is well known that for all $x,y\in\mathbb{R}^{n}$, we have 
\begin{equation}
\langle x,y\rangle\leq\gamma(x)\gamma^{\circ}(y).\label{eq: gen Cauchy-Schwartz}
\end{equation}
In fact, more is true and we have 
\begin{equation}
\gamma^{\circ}(y)=\underset{x\ne0}{\max}\frac{\langle x,y\rangle}{\gamma(x)}.\label{eq: gen Cauchy-Schwartz 2}
\end{equation}
For a proof of this, see page 54 of \citep{MR3155183}.

It is easy to see that the the strict convexity of $K$ (which means
that $\partial K$ does not contain any line segment) is equivalent
to the strict convexity of $\gamma$. By homogeneity of $\gamma$,
the latter is equivalent to 
\[
\gamma(x+y)<\gamma(x)+\gamma(y)
\]
when $x\ne cy$ and $y\ne cx$ for any $c\ge0$. 

The following three lemmas do not require any assumption about $\partial U$.
\begin{lem}
\label{lem: segment to the closest pt}Suppose $y$ is one of the
$d_{K}$-closest points on $\partial U$ to $x\in U$. Then $y$ is
a $d_{K}$-closest point on $\partial U$ to every point of $]x,y[$.
If in addition $\gamma$ is strictly convex, then $y$ is the unique
$d_{K}$-closest point on $\partial U$ to points of $]x,y[$.
\end{lem}
\begin{proof}
Let $z\in]x,y[$, and suppose to the contrary that there is $w\in\partial U-\{y\}$
such that 
\[
\gamma(z-w)<\gamma(z-y).
\]
Then we have 
\[
\gamma(x-w)\le\gamma(x-z)+\gamma(z-w)<\gamma(x-z)+\gamma(z-y)=\gamma(x-y).
\]
Which is a contradiction. 

Now suppose $\gamma$ is strictly convex, and 
\[
\gamma(z-w)\le\gamma(z-y).
\]
If $w$ belongs to the line containing $x,z,y$, then considering
the order of these four points on that line, we can easily arrive
at a contradiction. Hence, $x,z,w$ are not collinear, and by strict
convexity of $\gamma$ we get 
\[
\gamma(x-w)<\gamma(x-z)+\gamma(z-w)\le\gamma(x-z)+\gamma(z-y)=\gamma(x-y).
\]
Which is a contradiction too.
\end{proof}
\begin{lem}
\label{lem: d_K not C1}Suppose $\gamma$ is strictly convex. If $d_{K}(x)=\gamma(x-y)=\gamma(x-z)$
for two different points $y,z$ on $\partial U$, then $d_{K}$ is
not differentiable at $x$.
\end{lem}
\begin{proof}
The points in the segment $[x,y]$ have $y$ as $d_{K}$-closest point
on $\partial U$. Hence for $0\le t\le\gamma(x-y)$ we have 
\begin{align*}
d_{K}\big(x-\frac{t}{\gamma(x-y)}(x-y)\big) & =\gamma\big(x-\frac{t}{\gamma(x-y)}(x-y)-y\big)\\
 & =\big(1-\frac{t}{\gamma(x-y)}\big)\gamma(x-y)=\gamma(x-y)-t.
\end{align*}
Now suppose to the contrary that $d_{K}$ is differentiable at $x$.
Then by differentiating the above equality with respect to $t$ (and
the similar formula for $z$), we get 
\[
\big\langle Dd_{K}(x),\frac{x-y}{\gamma(x-y)}\big\rangle=1=\big\langle Dd_{K}(x),\frac{x-z}{\gamma(x-z)}\big\rangle.
\]

On the other hand, it is easy to show that $\gamma^{\circ}(Dd_{K}(x))\le1$.
To do this, just note that 
\[
d_{K}(x+tv)-d_{K}(x)\le\gamma(x+tv-x)=t\gamma(v).
\]
Taking the limit as $t\to0^{+}$, we get $\langle Dd_{K}(x),v\rangle\le\gamma(v)$.
We get the desired by (\ref{eq: gen Cauchy-Schwartz 2}).

Now note that there is at most one vector $v$ with $\gamma(v)=1$
such that 
\[
\langle Dd_{K}(x),v\rangle=1.
\]
Since, otherwise for two such vectors $v,w$, we would have $\langle Dd_{K}(x),(\frac{v+w}{2})\rangle=1$.
However, by strict convexity of $\gamma$, and inequality (\ref{eq: gen Cauchy-Schwartz}),
we get 
\begin{align*}
\big\langle Dd_{K}(x),\frac{v+w}{2}\big\rangle & \leq\gamma^{\circ}(Dd_{K}(x))\gamma\big(\frac{v+w}{2}\big)\\
 & <\gamma^{\circ}(Dd_{K}(x))\frac{\gamma(v)+\gamma(w)}{2}=1.
\end{align*}
Which is a contradiction. Therefore $d_{K}$ can not be differentiable
at $x$.
\end{proof}
\begin{defn}
For a strictly convex $K$, the subset of the $d_{K}$-ridge consisting
of the points with more than one $d_{K}$-closest point on $\partial U$,
is denoted by 
\[
R_{K,0}.
\]
\end{defn}
\begin{lem}
\label{lem: cont of y}Suppose $x_{i}\in\overline{U}$ converge to
$x\in\overline{U}$, and $y\in\partial U$ is the unique $d_{K}$-closest
point to $x$. If $y_{i}\in\partial U$ is a (not necessarily unique)
$d_{K}$-closest point to $x_{i}$, then $y_{i}$ converges to $y$.

If $x$ has more than one $d_{K}$-closest point on $\partial U$,
and $y_{i}$ converges to $\tilde{y}\in\partial U$, then $\tilde{y}$
is one of the $d_{K}$-closest points on $\partial U$ to $x$.
\end{lem}
\begin{proof}
Suppose that the claim of the first part does not hold. Then a subsequence
of $y_{i}$, which we still denote it by $y_{i}$, will remain outside
an open ball $B$ around $y$. Now consider the set $L:=x-d_{K}(x)K$
that touches $\partial U$ only at $y$. Since $L$ is a compact set
inside the open set $U\cup B$, a set of the form $x-(d_{K}(x)+\varepsilon)K$
is still inside $U\cup B$. Now, let $\epsilon<\frac{\varepsilon}{2},d_{K}(x)$.
As $x_{i}$'s approach $x$, they will be inside $x-\epsilon K$ eventually.
Therefore 
\[
d_{K}(x)+\varepsilon\le\gamma(x-y_{i})\le\gamma(x-x_{i})+\gamma(x_{i}-y_{i})<\epsilon+\gamma(x_{i}-y_{i}).
\]
Hence 
\[
d_{K}(x)+\frac{\varepsilon}{2}<d_{K}(x)+\varepsilon-\epsilon<d_{K}(x_{i}).
\]
But this contradicts the continuity of $d_{K}$.

Now let us consider the second statement. If the claim fails, then
$\tilde{y}$ is outside the compact set $L$. We can enlarge $L$
to $x-(d_{K}(x)+\varepsilon)K$ so that $\tilde{y}$ is still outside
the enlarged set. Now, let $\epsilon<\frac{\varepsilon}{3},d_{K}(x)$.
As $x_{i}\to x$ and $y_{i}\to\tilde{y}$, they will be respectively
inside $x-\epsilon K$ and $y-\epsilon K$ eventually. Thus 
\[
d_{K}(x)+\varepsilon\le\gamma(x-\tilde{y})\le\gamma(x-x_{i})+\gamma(x_{i}-y_{i})+\gamma(y_{i}-\tilde{y})<2\epsilon+\gamma(x_{i}-y_{i}).
\]
Which gives a contradiction as above.
\end{proof}

\subsection{Regularity of the gauge function}

Remember that $K$ is a compact convex subset of $\mathbb{R}^{n}$
whose interior contains the origin. Suppose that $\partial K$ is
$C^{k,\alpha}$ $(k\ge2\,,\,0\le\alpha\le1)$. Let us show that as
a result, $\gamma$ is $C^{k,\alpha}$ on $\mathbb{R}^{n}-\{0\}$.
Let $r=\rho(\theta)$ for $\theta\in\mathbb{S}^{n-1}$, be the equation
of $\partial K$ in polar coordinates. Then $\rho$ is positive and
$C^{k,\alpha}$. To see this note that locally, $\partial K$ is given
by a $C^{k,\alpha}$ equation $f(x)=0$. On the other hand we have
$x=rX(\theta)$, for some smooth function $X$. Hence we have $f(rX(\theta))=0$;
and the derivative of this expression with respect to $r$ is 
\[
\langle X(\theta),Df(rX(\theta))\rangle=\frac{1}{r}\langle x,Df(x)\rangle.
\]
But this is nonzero since $Df$ is orthogonal to $\partial K$, and
$x$ cannot be tangent to $\partial K$ (otherwise $0$ cannot be
in the interior of $K$, as $K$ lies on one side of its supporting
hyperplane at $x$). Thus we get the desired by the Implicit Function
Theorem. Now, it is straightforward to check that for a nonzero point
in $\mathbb{R}^{n}$ with polar coordinates $(s,\phi)$ we have 
\[
\gamma((s,\phi))=\frac{s}{\rho(\phi)}.
\]
This formula easily gives the smoothness of $\gamma$.
\begin{rem}
The above argument works when $k=1$ too, but we need the extra regularity
for what follows. Also note that as $\partial K=\{\gamma=1\}$ and
$D\gamma\ne0$ by (\ref{eq: g(Dg)=00003D1}), $\partial K$ is as
smooth as $\gamma$.
\end{rem}
Now, suppose in addition that $K$ is strictly convex. Then $\gamma$
is strictly convex too. By Remark 1.7.14 and Theorem 2.2.4 of \citep{MR3155183},
$K^{\circ}$ is also strictly convex and its boundary is $C^{1}$.
Therefore $\gamma^{\circ}$ is strictly convex, and it is $C^{1}$
on $\mathbb{R}^{n}-\{0\}$. Thence by Corollary 1.7.3 of \citep{MR3155183},
for $x\ne0$ we have 
\begin{eqnarray}
D\gamma(x)\in\partial K^{\circ}, &  & D\gamma^{\circ}(x)\in\partial K.\label{eq: g(Dg)=00003D1}
\end{eqnarray}
In particular $D\gamma,D\gamma^{\circ}$ are nonzero on $\mathbb{R}^{n}-\{0\}$.

We also suppose that the smallest principal curvature of $\partial K$
is positive everywhere except possibly at a finite number of points
where it vanishes. Let $\{\mu_{1},\dots,\mu_{m}\}$ be the outward
unit normal to $\partial K$ at these points.

We can show that $\gamma^{\circ}$ is $C^{k,\alpha}$ on $\mathbb{R}^{n}-\{t\mu_{i}:t\ge0\,,\,i=1,\dots,m\}$.
To see this, let $n_{K}:\partial K\to\mathbb{S}^{n-1}$ be the Gauss
map, i.e. $n_{K}(y)$ is the outward unit normal to $\partial K$
at $y$. Then $n_{K}$ is $C^{k-1,\alpha}$ and its derivative is
an isomorphism at the points with positive principal curvatures. Hence
$n_{K}$ is locally invertible with a $C^{k-1,\alpha}$ inverse $n_{K}^{-1}$,
around any point of $\mathbb{S}^{n-1}-\{\mu_{1},\dots,\mu_{m}\}$.
Now note that as it is well known, $\gamma^{\circ}$ equals the support
function of $K$, i.e. 
\[
\gamma^{\circ}(x)=\sup\{\langle x,y\rangle:y\in K\}.
\]
Thus as shown on page 115 of \citep{MR3155183}, for $x\ne0$ we have
\[
D\gamma^{\circ}(x)=n_{K}^{-1}(\frac{x}{|x|}).
\]
Which gives the desired result. As a consequence, since $\partial K^{\circ}=\{\gamma^{\circ}=1\}$
and $D\gamma^{\circ}\ne0$ by (\ref{eq: g(Dg)=00003D1}), $\partial K^{\circ}$
is $C^{k,\alpha}$ except possibly at finitely many points which are
positive multiples of $\mu_{i}$'s.

Let us recall a few more properties of $\gamma,\gamma^{\circ}$. Since
they are positively 1-homogenous, $D\gamma,D\gamma^{\circ}$ are positively
0-homogenous, and $D^{2}\gamma,D^{2}\gamma^{\circ}$ (the latter when
exists) are positively $(-1)$-homogenous, i.e. 
\begin{eqnarray}
\gamma(tx)=t\gamma(x), & D\gamma(tx)=D\gamma(x), & D^{2}\gamma(tx)=\frac{1}{t}D^{2}\gamma(x),\nonumber \\
\gamma^{\circ}(tx)=t\gamma^{\circ}(x), & D\gamma^{\circ}(tx)=D\gamma^{\circ}(x), & D^{2}\gamma^{\circ}(tx)=\frac{1}{t}D^{2}\gamma^{\circ}(x),\label{eq: homog}
\end{eqnarray}
for $x\ne0$ and $t>0$. As a result, using Euler's theorem on homogenous
functions we get 
\begin{eqnarray}
\langle D\gamma(x),x\rangle=\gamma(x), &  & D^{2}\gamma(x)\,x=0,\nonumber \\
\langle D\gamma^{\circ}(x),x\rangle=\gamma^{\circ}(x), &  & D^{2}\gamma^{\circ}(x)\,x=0,\label{eq: Euler formula}
\end{eqnarray}
for $x\ne0$. Note that in both (\ref{eq: homog}), (\ref{eq: Euler formula})
we need to assume $x\ne t\mu_{i}$ for any $t>0$, when dealing with
$D^{2}\gamma^{\circ}$. We also recall the following fact from \citep{MR2336304},
that for $x\ne0$ 
\begin{eqnarray}
D\gamma^{\circ}(D\gamma(x))=\frac{x}{\gamma(x)}, &  & D\gamma(D\gamma^{\circ}(x))=\frac{x}{\gamma^{\circ}(x)}.\label{eq: Dg(Dg0)}
\end{eqnarray}
\begin{rem}
Let us assume for simplicity that $n=2$. As a consequence of (\ref{eq: Euler formula}),
we see that if $x\ne t\mu_{i}$ for any $t>0$, then it is an eigenvector
of $D^{2}\gamma^{\circ}(x)$ with eigenvalue $0$. Since $D^{2}\gamma^{\circ}(x)$
is a symmetric matrix, its other eigenvector can be taken to be $x^{\perp}$.
By Corollary 2.5.2 of \citep{MR3155183} and $(-1)$-homogeneity of
$D^{2}\gamma^{\circ}$, the other eigenvalue of $D^{2}\gamma^{\circ}(x)$
is 
\begin{equation}
\frac{1}{|x|}r_{K}(n_{K}^{-1}(\frac{x}{|x|})).\label{eq: eigvlu of D2g}
\end{equation}
Here $r_{K}$ is the radius of curvature of $\partial K$, i.e. the
reciprocal of its curvature; and $n_{K}^{-1}$ is the inverse of the
Gauss map of $\partial K$. Hence, the eigenvalues of $D^{2}\gamma^{\circ}(x)$
are $0$ and a positive number.
\end{rem}

\section{\label{sec: Regularity d}Regularity of the distance function}

In this section, we are going to study the singularities of the function
$d_{K}$. It is obvious that $d_{K}$ is a Lipschitz function. We
want to characterize the set over which it is more regular. In order
to do that, we need to impose some restrictions on $\partial K,\partial U$.

For the rest of this paper we assume that $n=2$. Let $U\subset\mathbb{R}^{2}$
be a bounded open set, whose boundary is the union of simple closed
Jordan curves consisting of arcs $S_{1},\dots,S_{N}$ which are $C^{k,\alpha}$
$(k\ge2\,,\,0\le\alpha\le1)$ up to their endpoints, satisfying Assumption
\ref{assu: 1} below. Thus, topologically, $U$ is homeomorphic to
the interior of a disk from which, possibly, several disks are removed.
If $S_{i}\cap S_{j}$ is nonempty, in which case it consists of a
single point, we call that point a corner or a vertex of $\partial U$.
A \textbf{nonreentrant} corner of $\partial U$ is a corner whose
opening angle is less than $\pi$. And, a \textbf{reentrant} corner
is a corner with opening angle greater than or equal to $\pi$. If
the angle of a reentrant corner is strictly greater than $\pi$ we
call it a \textbf{strict reentrant} corner. We assume that the opening
angles of the vertices of $\partial U$ are strictly between $0$
and $2\pi$, i.e. there are no \textit{cusps}. As a result, $\partial U$
is locally the graph of a Lipschitz function.
\begin{rem}
We can allow cusps with angle $0$ in Theorem \ref{thm: d_K is C^2},
and arbitrary cusps in Theorems \ref{thm: d_K is C^2, corners}, \ref{thm: ridge}
and \ref{thm: 1-kd>0}. But we need the Lipschitz regularity of $\partial U$
when we deal with the variational problem.
\end{rem}
\begin{assumption}
\label{assu: 1}\textcolor{blue}{{} }Let $y\in S_{i}$ be an interior
point of $S_{i}$, or a reentrant corner. We assume that if the inward
unit normal to $S_{i}$ at $y$ belongs to $\{\mu_{1},\dots,\mu_{m}\}$,
then either the curvature of $S_{i}$ at $y$ is positive, or $S_{i}$
is a line segment.
\end{assumption}
Note that there are at most finitely many points on each $S_{i}$
at which the inward unit normal belongs to $\{\mu_{1},\dots,\mu_{m}\}$,
and the curvature of $S_{i}$ at them is positive. The reason is that
these points are isolated; because the derivative of the inward normal
at them is nonzero, due to the positivity of the curvature (see (\ref{eq: D2g.v'})).

First we assume that all the corners of $\partial U$ are nonreentrant.
We will consider domains with reentrant corners later.

Next, we introduce a new notion of curvature for curves in the plane.
It will be used to study the regularity of $d_{K}$.
\begin{defn}
The \textbf{$\boldsymbol{K}$-curvature} of a $C^{2}$ curve $t\mapsto(x(t),y(t))$
in the plane is 
\[
\kappa_{K}:=\frac{1}{|\nu|^{2}}\langle D^{2}\gamma^{\circ}(\nu)\,\nu',\nu^{\perp}\rangle.
\]
Here, $\nu:=(-y',x')$ is normal to the curve, and $D^{2}\gamma^{\circ}(\nu)\,\nu'$
is the action of the matrix $D^{2}\gamma^{\circ}(\nu)$ on the vector
$\nu'$. We assume that $\nu$ is nonzero and is not a positive multiple
of any of $\mu_{i}$'s. When the curve is a line segment and $\nu\equiv c\mu_{i}$
for some $c>0$, we define $\kappa_{K}\equiv0$.
\end{defn}
It is easy to see that $\kappa_{K}$ does not change under reparametrizations
of the curve, hence it is an intrinsic quantity. Also note that $\langle\nu',\nu^{\perp}\rangle=\kappa|\nu|^{3}$,
where $\kappa$ is the ordinary curvature.
\begin{lem}
We have 
\begin{eqnarray}
 &  & D^{2}\gamma^{\circ}(\nu)\,\nu'=\kappa_{K}\nu^{\perp},\nonumber \\
 &  & \kappa_{K}=\frac{1}{\gamma^{\circ}(\nu)}\langle D^{2}\gamma^{\circ}(\nu)\,\nu',D^{\perp}\gamma^{\circ}(\nu)\rangle.\label{eq: D2g.v'}
\end{eqnarray}
\end{lem}
\begin{proof}
Since we have $D^{2}\gamma^{\circ}(\nu)\,\nu=0$, and $D^{2}\gamma^{\circ}$
is a symmetric matrix, we get 
\[
\langle(D^{2}\gamma^{\circ}\,\nu')^{\perp},\nu^{\perp}\rangle=\langle D^{2}\gamma^{\circ}\,\nu',\nu\rangle=\langle\nu',D^{2}\gamma^{\circ}\,\nu\rangle=0.
\]
Thus $D^{2}\gamma^{\circ}(\nu)\,\nu'$ is parallel to $\nu^{\perp}$,
and from the definition of $K$-curvature we get $D^{2}\gamma^{\circ}(\nu)\,\nu'=\kappa_{K}\nu^{\perp}$. 

Then by (\ref{eq: Euler formula}) we get
\[
\langle D^{2}\gamma^{\circ}\,\nu',D^{\perp}\gamma^{\circ}\rangle=\langle\kappa_{K}\nu^{\perp},D^{\perp}\gamma^{\circ}\rangle=\kappa_{K}\langle\nu,D\gamma^{\circ}\rangle=\kappa_{K}\gamma^{\circ}(\nu).
\]
\end{proof}
\begin{lem}
\label{lem: sign k}$\kappa_{K}$ has the same sign as the ordinary
curvature $\kappa$. In particular, $\kappa_{K}=0$ if and only if
$\kappa=0$.
\end{lem}
\begin{proof}
We can write $\nu'$ as a linear combination of $\nu,\nu^{\perp}$
\[
\nu'=a\nu+b\nu^{\perp}.
\]
Since by (\ref{eq: eigvlu of D2g}) we know that $D^{2}\gamma^{\circ}(\nu)\,\nu^{\perp}=\lambda\nu^{\perp}$
for some $\lambda>0$, using (\ref{eq: Euler formula}),(\ref{eq: D2g.v'})
we get 
\[
\kappa_{K}\nu^{\perp}=D^{2}\gamma^{\circ}(\nu)\,\nu'=\lambda b\nu^{\perp}.
\]
On the other hand $\kappa=\frac{\langle\nu',\nu^{\perp}\rangle}{|\nu|^{3}}=\frac{b}{|\nu|}$.
Therefore 
\[
\kappa_{K}=|\nu|\lambda\kappa.
\]
\end{proof}
\begin{rem}
By (\ref{eq: eigvlu of D2g}), the interpretation of the above formula
is that the $K$-curvature at a point with normal $\nu$, is the ordinary
curvature at that point divided by the ordinary curvature of $\partial K$
at the unique point with outward normal $\nu$.
\end{rem}
\begin{thm}
\label{thm: d_K is C^2}Suppose $K\subset\mathbb{R}^{2}$ is a compact
strictly convex set with zero in its interior, such that $\partial K$
is $C^{k,\alpha}$ $(k\ge2\,,\,0\le\alpha\le1)$, with positive\textcolor{red}{{}
}\textcolor{black}{curvature} except at a finite number of points.
Also suppose that $U\subset\mathbb{R}^{2}$ is a bounded open set,
with piecewise $C^{k,\alpha}$ boundary which satisfies Assumption
\ref{assu: 1}, and only has nonreentrant corners. Let $x\in U-R_{K,0}$,
and let $y=y(x)$ be the unique $d_{K}$-closest point to $x$ on
$\partial U$. If 
\[
\kappa_{K}(y(x))d_{K}(x)\ne1,
\]
then $d_{K}=d_{K}(\cdot,\partial U)$ is $C^{k,\alpha}$ around $x$.
Furthermore, if $\nu$ is an inward normal to $\partial U$ at $y$,
and $\zeta$ is a unit vector orthogonal to the segment $]x,y[$,
we have 
\begin{eqnarray}
 &  & Dd_{K}(x)=\frac{\nu}{\gamma^{\circ}(\nu)},\nonumber \\
 &  & \Delta d_{K}(x)=\frac{-\kappa(y)|\nu|^{3}|D\gamma^{\circ}(\nu)|^{2}}{\gamma^{\circ}(\nu)^{3}(1-\kappa_{K}(y)d_{K}(x))},\nonumber \\
 &  & D_{vw}^{2}d_{K}(x)=\Delta d_{K}(x)\langle v,\zeta\rangle\langle w,\zeta\rangle.\label{eq: laplace d_K}
\end{eqnarray}
Here, $\kappa$ is the ordinary curvature, and $\kappa_{K}$ is the
$K$-curvature of $\partial U$; and $v,w$ are arbitrary vectors
in $\mathbb{R}^{2}$.
\end{thm}
\begin{proof}
The set $L:=x-d_{K}(x)K$ is inside $\overline{U}$ and touches $\partial U$
only at $y$. Since $\partial K$ is $C^{1}$, $y$ is not a nonreentrant
corner of $\partial U$.

Let $\nu$ be an inward normal to $\partial U$. Note that $\nu(y)$
is also an inward normal to $\partial L$ at $y$. We claim that 
\begin{equation}
\frac{x-y}{\gamma(x-y)}=D\gamma^{\circ}(\nu(y)).\label{eq: K-normal}
\end{equation}
Note that $\xi:=\frac{x-y}{\gamma(x-y)}\in\partial K$. Hence by (\ref{eq: Dg(Dg0)})
we have 
\[
D\gamma^{\circ}(D\gamma(\xi))=\xi.
\]
But $D\gamma(\xi)$, which is nonzero, is an outward normal to $\partial K$
at $\xi$. The reason is that $\partial K=\{\gamma=1\}$, and $\gamma$
increases as we move to the outside of $K$. On the other hand, $-\nu(y)$
is an inward normal to $\partial(x-L)$ at $x-y$, and consequently
an inward normal to $\partial K$ at $\xi$. Hence due to the positive
0-homogeneity of $D\gamma^{\circ}$ we get (\ref{eq: K-normal}).
Note that $\nu$ need not be unit for (\ref{eq: K-normal}) to hold. 

As a consequence of (\ref{eq: K-normal}), we have 
\begin{equation}
x=y(x)+d_{K}(x)\,D\gamma^{\circ}(\nu(y)).\label{eq: parametrize by d_K}
\end{equation}
Note that (\ref{eq: K-normal}) holds even if $x\in R_{K,0}$ and
$y$ is one of the $d_{K}$-closest points to $x$ on $\partial U$
(or even when $y$ is a reentrant corner and $\nu(y)$ is an inward
normal to $\partial L$ at $y$). Thus, formula (\ref{eq: parametrize by d_K})
holds in these cases too.

Let us show that if $\frac{\nu(y)}{|\nu(y)|}\in\{\mu_{1},\dots,\mu_{m}\}$,
then $\kappa(y)\le0$. Thus by Assumption \ref{assu: 1}, $\partial U$
must be a line segment around $y$. To see this, note that $L$ is
tangent to $\partial U$ at $y$, and $L-\{y\}\subset U$. This implies
that the curvature of $\partial L$ at $y$, which is zero, cannot
be less than the curvature of $\partial U$ at $y$.

Let $t\mapsto(y_{1}(t),y_{2}(t))$ for $|t|<\beta$ be a smooth nondegenerate
parametrization of $\partial U$ around $y$, with $(y_{1}(0),y_{2}(0))=y$.
Also suppose that the direction of the parametrization is such that
$\nu(t):=(-y_{2}'(t),y_{1}'(t))$ is an inward normal to $\partial U$.
We can take $\beta$ small enough to ensure that by Assumption \ref{assu: 1}
and the above paragraph, $\nu$ is not a positive multiple of any
of $\mu_{i}$'s unless it is constant. 

Consider the map 
\[
F\,:\,(t,d)\mapsto(y_{1}(t),y_{2}(t))+d\,D\gamma^{\circ}(-y_{2}'(t),y_{1}'(t))
\]
from the open set $(-\beta,\beta)\times(0,\infty)$ into $\mathbb{R}^{2}$.
We have $F(0,d_{K}(x))=x$. We wish to compute $DF$ around this point.
Note that $D\gamma^{\circ}(\nu(t))$ is differentiable with respect
to $t$. Now we have 
\[
DF(t,d)=\begin{bmatrix}y_{1}'+[-y_{2}''D_{11}^{2}\gamma^{\circ}+y_{1}''D_{12}^{2}\gamma^{\circ}]d &  & D_{1}\gamma^{\circ}\\
\\
y_{2}'+[-y_{2}''D_{12}^{2}\gamma^{\circ}+y_{1}''D_{22}^{2}\gamma^{\circ}]d &  & D_{2}\gamma^{\circ}
\end{bmatrix}.
\]
Consequently 
\begin{eqnarray*}
 & \det DF & =-y_{2}'D_{1}\gamma^{\circ}+y_{1}'D_{2}\gamma^{\circ}\\
 &  & \qquad-\,d\,\big[-y_{2}''(D_{1}\gamma^{\circ}D_{12}^{2}\gamma^{\circ}-D_{2}\gamma^{\circ}D_{11}^{2}\gamma^{\circ})+y_{1}''(D_{1}\gamma^{\circ}D_{22}^{2}\gamma^{\circ}-D_{2}\gamma^{\circ}D_{12}^{2}\gamma^{\circ})\big]\\
 &  & =\langle\nu,D\gamma^{\circ}(\nu)\rangle-d\,\langle D^{\perp}\gamma^{\circ}(\nu),D^{2}\gamma^{\circ}(\nu)\,\nu'\rangle\\
 &  & =\gamma^{\circ}(\nu)(1-\kappa_{K}d).
\end{eqnarray*}
Here, we used (\ref{eq: Euler formula}), (\ref{eq: D2g.v'}).

Now if we assume that $\kappa_{K}(y(x_{0}))d_{K}(x_{0})\ne1$ for
some $x_{0}\in U-R_{K,0}$, then $F$ is $C^{k-1,\alpha}$ around
$(0,d_{K}(x_{0}))$ with a $C^{k-1,\alpha}$ inverse. Since $F\,:\,(t,d)\mapsto x$
is invertible in a neighborhood of $(0,d_{K}(x_{0}))$, we have 
\begin{equation}
x=F(t(x),d(x))=y(t(x))+d(x)\,D\gamma^{\circ}(\nu(t(x))).
\end{equation}
We also know that in general 
\[
x=y(x)+d_{K}(x)\,D\gamma^{\circ}(\nu(y(x))).
\]
If we take $x$ close enough to $x_{0}$, then by continuity $y(x),d_{K}(x)$
will be close to $y(x_{0}),d_{K}(x_{0})$ (here we use Lemma \ref{lem: cont of y}
and the fact that $x\notin R_{K,0}$), and by invertibility of $F$
we get 
\begin{eqnarray*}
y(x)=y(t(x)), &  & d_{K}(x)=d(x).
\end{eqnarray*}
As we showed that $x\mapsto(t,d)$ is locally $C^{k-1,\alpha}$, we
obtain that $d_{K}(x)$ and $y(x)$ are also locally $C^{k-1,\alpha}$. 

Note that the above also shows that all points around $x_{0}$ have
a unique $d_{K}$-closest point around $y(x_{0})$, which by continuity
is the unique $d_{K}$-closest point to them on $\partial U$. Thus,
a neighborhood of $x_{0}$ is in $U-R_{K,0}$. This can also be seen
from the fact that $d_{K}$ is differentiable around $x_{0}$.

We can easily compute 
\[
DF^{-1}=\frac{1}{\gamma^{\circ}(\nu)(1-\kappa_{K}d)}\begin{bmatrix}D_{2}\gamma^{\circ} &  & -D_{1}\gamma^{\circ}\\
\\
-y_{2}'-[-y_{2}''D_{12}^{2}\gamma^{\circ}+y_{1}''D_{22}^{2}\gamma^{\circ}]d &  & y_{1}'+[-y_{2}''D_{11}^{2}\gamma^{\circ}+y_{1}''D_{12}^{2}\gamma^{\circ}]d
\end{bmatrix}.
\]
Using (\ref{eq: D2g.v'}) we can simplify this as 
\[
DF^{-1}=\frac{1}{\gamma^{\circ}(\nu)(1-\kappa_{K}d)}\begin{bmatrix}-D^{\perp}\gamma^{\circ}\\
\nu+d(D^{2}\gamma^{\circ}\,\nu')^{\perp}
\end{bmatrix}=\begin{bmatrix}-\frac{D^{\perp}\gamma^{\circ}}{\gamma^{\circ}(\nu)(1-\kappa_{K}d)}\\
\frac{\nu}{\gamma^{\circ}(\nu)}
\end{bmatrix}.
\]
Which implies 
\begin{eqnarray*}
 &  & Dd_{K}(x)=\frac{\nu}{\gamma^{\circ}(\nu)}=\frac{\nu(t(x))}{\gamma^{\circ}(\nu(t(x)))},\\
 &  & Dt(x)=-\frac{D^{\perp}\gamma^{\circ}(\nu)}{\gamma^{\circ}(\nu)(1-\kappa_{K}(y)d_{K}(x))}.
\end{eqnarray*}
Consequently, since $\nu,t$ are $C^{k-1,\alpha}$ functions and $\gamma^{\circ}$
is $C^{k,\alpha}$ on the image of $\nu$ (otherwise $\frac{\nu}{\gamma^{\circ}(\nu)}$
is constant), $d_{K}$ is $C^{k,\alpha}$. 

By differentiating $d_{K}$ one more time, for $i=1,2$ we get 
\[
D_{ii}d_{K}=\Big[\frac{\nu_{i}'}{\gamma^{\circ}(\nu)}-\frac{\nu_{i}\langle D\gamma^{\circ}(\nu),\nu'\rangle}{\gamma^{\circ}(\nu)^{2}}\Big]D_{i}t.
\]
Hence
\begin{eqnarray*}
 & \Delta d_{K} & =[\nu_{1}'\gamma^{\circ}(\nu)-\nu_{1}\langle D\gamma^{\circ}(\nu),\nu'\rangle]\frac{D_{2}\gamma^{\circ}(\nu)}{\gamma^{\circ}(\nu)^{3}(1-\kappa_{K}d_{K})}\\
 &  & \qquad-\,[\nu_{2}'\gamma^{\circ}(\nu)-\nu_{2}\langle D\gamma^{\circ}(\nu),\nu'\rangle]\frac{D_{1}\gamma^{\circ}(\nu)}{\gamma^{\circ}(\nu)^{3}(1-\kappa_{K}d_{K})}\\
 &  & =-\frac{\langle D^{\perp}\gamma^{\circ},\nu'\rangle\gamma^{\circ}(\nu)-\langle D\gamma^{\circ},\nu'\rangle\langle D^{\perp}\gamma^{\circ},\nu\rangle}{\gamma^{\circ}(\nu)^{3}(1-\kappa_{K}d_{K})}.
\end{eqnarray*}
Now as $\gamma^{\circ}(\nu)=\langle D\gamma^{\circ},\nu\rangle$,
the numerator of the above fraction can be written as 
\[
\langle v^{\perp},\nu'\rangle\langle v,\nu\rangle-\langle v,\nu'\rangle\langle v^{\perp},\nu\rangle,
\]
where $v:=D\gamma^{\circ}(\nu)$. Since $v,v^{\perp}$ are orthogonal
and have the same length, this expression is nothing but $|v|^{2}\langle\nu',\nu^{\perp}\rangle$.
Therefore using the fact that $\langle\nu',\nu^{\perp}\rangle=\kappa|\nu|^{3}$
we get the desired result.

Now, let $\tilde{\xi}:=\frac{x-y(x)}{|x-y(x)|}$, and $\zeta:=-\tilde{\xi}^{\perp}$.
Then as $Dd_{K}$ is constant along the segment $]x,y(x)[$, we have
$D_{\tilde{\xi}\tilde{\xi}}^{2}d_{K}(x)=D_{\tilde{\xi}\zeta}^{2}d_{K}(x)=0$.
Also as $\tilde{\xi},\zeta$ form an orthonormal basis, we have 
\[
\Delta d_{K}(x)=D_{\tilde{\xi}\tilde{\xi}}^{2}d_{K}(x)+D_{\zeta\zeta}^{2}d_{K}(x)=D_{\zeta\zeta}^{2}d_{K}(x).
\]
Therefore, by changing the coordinates from the orthonormal basis
$\tilde{\xi},\zeta$ to the standard basis, we get 
\[
D^{2}d_{K}(x)=\begin{bmatrix}\tilde{\xi}_{1} & \zeta_{1}\\
\tilde{\xi}_{2} & \zeta_{2}
\end{bmatrix}\begin{bmatrix}0 & 0\\
0 & \Delta d_{K}(x)
\end{bmatrix}\begin{bmatrix}\tilde{\xi}_{1} & \tilde{\xi}_{2}\\
\zeta_{1} & \zeta_{2}
\end{bmatrix}.
\]
By applying both sides of this equation to two vectors $v=(v_{1},v_{2})$,
$w=(w_{1},w_{2})$, we obtain (\ref{eq: laplace d_K}).
\end{proof}

\subsection{Domains with reentrant corners}

Now, we allow some of the vertices of $\partial U$ to be reentrant
corners. The main difference with the previous case, is that reentrant
corners can be the $d_{K}$-closest point on $\partial U$ to some
points in $U$. Let us first introduce a new notion.
\begin{defn}
The inward \textbf{$\boldsymbol{K}$-normal} at a point $y\in S_{i}\subset\partial U$
is 
\[
\nu_{K}(y):=D\gamma^{\circ}(\nu(y)),
\]
where $\nu(y)$ is an inward normal to $S_{i}$ at $y$.
\end{defn}
The value of $\nu_{K}$ is independent of the length of $\nu$ due
to the 0-homogeneity of $D\gamma^{\circ}$. Also, we have $\gamma(\nu_{K})=1$
and 
\[
\langle\nu_{K},\nu\rangle=\gamma^{\circ}(\nu)>0,
\]
by (\ref{eq: Euler formula}), (\ref{eq: g(Dg)=00003D1}). In particular,
$\nu_{K}$ is really pointing inward. Note that at a corner we have
two inward $K$-normals.

The motivation for this definition is that by (\ref{eq: K-normal}),
$\nu_{K}(y)$ is the direction along which points inside $U$ and
close to $y$ have $y$ as the $d_{K}$-closest point on $\partial U$,
if $y$ is the $d_{K}$-closest point to any point inside $U$. 
\begin{thm}
\label{thm: d_K is C^2, corners}Suppose $K\subset\mathbb{R}^{2}$
is a compact strictly convex set with zero in its interior, such that
$\partial K$ is $C^{k,\alpha}$ $(k\ge2\,,\,0\le\alpha\le1)$, with
positive\textcolor{red}{{} }\textcolor{black}{curvature} except at a
finite number of points. Also suppose that $U\subset\mathbb{R}^{2}$
is a bounded open set, with piecewise $C^{k,\alpha}$ boundary which
satisfies Assumption \ref{assu: 1}. Let $x\in U-R_{K,0}$, and let
$y=y(x)$ be the unique $d_{K}$-closest point to $x$ on $\partial U$.
If $y$ is not a reentrant corner and 
\[
\kappa_{K}(y(x))d_{K}(x)\ne1,
\]
then $d_{K}=d_{K}(\cdot,\partial U)$ is $C^{k,\alpha}$ around $x$.
Furthermore, $Dd_{K}$,$D^{2}d_{K}$ at $x$ are given by (\ref{eq: laplace d_K}).

If $y$ is a strict reentrant corner and $x-y$ is not parallel to
one of the inward $K$-normals at $y$, then 
\[
d_{K}(z)=\gamma(z-y),
\]
for $z$ close to $x$. Thus $d_{K}$ is $C^{k,\alpha}$ around $x$.
And, if $x-y$ is parallel to one of the inward $K$-normals at $y$
and $\kappa_{K}(y)d_{K}(x)\neq1$, where $\kappa_{K}$ is the $K$-curvature
of the corresponding boundary part, then $d_{K}$ is $C^{1,1}$ around
$x$ (but not $C^{2}$ in general).

Finally, if $y$ is a non-strict reentrant corner and $d_{K}(x)\neq\frac{1}{\kappa_{K,1}}\,,\,\frac{1}{\kappa_{K,2}}$,
where $\kappa_{K,1},\kappa_{K,2}$ are the $K$-curvatures at $y$
from different sides, then $d_{K}$ is $C^{1,1}$ around $x$ (but
not $C^{2}$ in general).
\end{thm}
\begin{proof}
If $y$ is not a reentrant corner, the proof is the same as in Theorem
\ref{thm: d_K is C^2}; so we assume that $y\in S_{1}\cap S_{2}$
is a reentrant corner. Consider the set $L:=x-d_{K}(x)K$ which is
inside $\overline{U}$ and touches $\partial U$ only at $y$. Note
that $y\in\partial L$. Let $\nu$ be the inward unit normal to $\partial L$
at $y$.

First suppose that $y$ is a strict reentrant corner. Let $\nu_{1},\nu_{2}$
be the inward unit normals to $S_{1},S_{2}$ at $y$. Then, $\nu$
must lie between $\nu_{1},\nu_{2}$ or coincide with one of them,
otherwise $L$ would intersect the exterior of $U$. If $x-y$ is
not parallel to one of the inward $K$-normals at $y$, then $\nu\neq\nu_{1},\nu_{2}$
by (\ref{eq: K-normal}). We need to show that 
\[
d_{K}(z)=\gamma(z-y),
\]
for $z$ close to $x$. 

To prove this, it is enough to show that $L_{z}:=z-\gamma(z-y)K$
is a subset of $\overline{U}$ for $z$ close to $x$. Suppose to
the contrary that there exists a sequence $z_{i}\to x$ such that
$L_{z_{i}}$ intersects $\mathbb{R}^{2}-\overline{U}$ at $y_{i}$.
Due to the compactness of $K$ we can assume that $y_{i}$ converges
to some limit. But that limit must belong to $L$, and it cannot be
an interior point of $U$; hence we must have $y_{i}\to y$. On the
other hand, $L_{z_{i}}$ lies on one side of the tangent line to $L_{z_{i}}$
at $y$, and that line is close to $l_{y}$, the tangent line to $L$
at $y$. Now, consider two half-lines with vertex $y$ which are between
$l_{y}$ and $S_{1},S_{2}$ respectively. Then for large enough $i$,
$L_{z_{i}}$ and $L$ are on the same side of the union of these two
half-lines. But this contradicts the fact that $y_{i}$ is in the
intersection of a neighborhood of $y$ and $\mathbb{R}^{2}-\overline{U}$.

Next consider the case where $\nu=\nu_{1}$. Then $x-y$ is parallel
to the $K$-normal to $S_{1}$ at $y$. Note that if $\nu_{1}$ coincides
with one of the $\mu_{i}$'s, then $S_{1}$ must be a line segment
by Assumption \ref{assu: 1}; otherwise $L$ cannot be tangent to
$S_{1}$ at $y$ and lies inside $\overline{U}$. Consider a small
ball around $x$ divided by $l$, the line passing through $x,y$.
Denote by $B$ the open side of the ball which is in the same side
of $l$ as $S_{1}$. First note that by Lemma \ref{lem: cont of y},
the $d_{K}$-closest points on $\partial U$ to points in $B$ must
be close to $y$; so they either lie on $S_{1}$ or $S_{2}$. But
if $B$ is small enough, those $d_{K}$-closest points cannot belong
to $S_{2}$. 

To see this, suppose to the contrary that $w_{i}\in S_{2}$ is $d_{K}$-closest
to $z_{i}\in B$, and $z_{i}\to x$. First let us assume that $w_{i}\in S_{2}-\{y\}$.
Then, by Lemma \ref{lem: cont of y} we know that $w_{i}\to y$. Also
by (\ref{eq: K-normal}) we have 
\[
\frac{z_{i}-w_{i}}{\gamma(z_{i}-w_{i})}=D\gamma^{\circ}(\nu(w_{i})).
\]
The left hand side of this equality converges to $\frac{x-y}{\gamma(x-y)}$
which equals $D\gamma^{\circ}(\nu_{1})$, while the right hand side
converges to $D\gamma^{\circ}(\nu_{2})$. Now, Corollary 1.7.3 of
\citep{MR3155183} says that for some unit vector $\tilde{\nu}$,
$D\gamma^{\circ}(\tilde{\nu})$ is the unique point on $\partial K$
which has $\tilde{\nu}$ as the outward unit normal. Since $\partial K$
is $C^{1}$, this implies that $D\gamma^{\circ}$ is injective on
the unit circle; thus we arrive at a contradiction.

Now let us show that $w_{i}$ cannot equal $y$ for any $i$. If this
happens, the definition of $B$ and (\ref{eq: K-normal}) imply that
the inward unit normal to $L_{z_{i}}$ at $y$, $\nu(w_{i})$, lies
between $\nu_{1}$ and $-\nu_{1}^{\perp}$. The reason is that $D\gamma^{\circ}$
is orientation preserving on the unit circle due to the convexity
of $\gamma^{\circ}$. In other words 
\[
\langle D\gamma^{\circ}(\nu(w_{i}))-D\gamma^{\circ}(\nu_{1}),\nu(w_{i})-\nu_{1}\rangle\ge0.
\]
Now, if $\nu(w_{i})=\nu_{1}$, then $x,y,z_{i}$ must be collinear
by (\ref{eq: parametrize by d_K}), which is impossible by the definition
of $B$; and if $\nu(w_{i})\ne\nu_{1}$, then $L_{z_{i}}$ would intersect
the exterior of $U$.

Thus far, we have shown that $B$ can be taken to be small enough
so that the $d_{K}$-closest points on $\partial U$ to points in
$B$ are on $S_{1}-\{y\}$. Let us also show that if $B$ is small
enough, then $R_{K,0}$ does not intersect it. Suppose to the contrary
that there is a sequence $z_{i}\to x$ of elements of $B$ such that
they all have more than one $d_{K}$-closest points on $S_{1}-\{y\}$.
Let $w_{i,1},w_{i,2}$ be two distinct $d_{K}$-closest points to
$z_{i}$. First note that for this to happen, $S_{1}$ cannot be a
line segment; since $K$ is strictly convex. Hence we can assume that
$\nu_{1}$ is not one of the $\mu_{j}$'s. Now, by (\ref{eq: parametrize by d_K})
we have 
\begin{eqnarray*}
z_{i}-w_{i,1}=d_{K}(z_{i})\,D\gamma^{\circ}(\nu(w_{i,1})), &  & z_{i}-w_{i,2}=d_{K}(z_{i})\,D\gamma^{\circ}(\nu(w_{i,2})).
\end{eqnarray*}
If we subtract these two equations we get 
\begin{equation}
w_{i,1}-w_{i,2}=-d_{K}(z_{i})\,[D\gamma^{\circ}(\nu(w_{i,1}))-D\gamma^{\circ}(\nu(w_{i,2}))].\label{eq: 1 in reentrant}
\end{equation}
Let $t\mapsto y(t)$ be a smooth nondegenerate parametrization of
$S_{1}$ around $y$ with $y(0)=y$. Then there are $t_{i,j}$ such
that $w_{i,j}=y(t_{i,j})$. Since $w_{i,1},w_{i,2}\to y$, we have
$t_{i,1},t_{i,2}\to0^{+}$. As $D\gamma^{\circ}$ is differentiable
at $\nu_{1}$, we can divide by $t_{i,1}-t_{i,2}$ and let $i\to\infty$
in (\ref{eq: 1 in reentrant}) to get 
\[
y'(0)=-d_{K}(x)[D^{2}\gamma^{\circ}(\nu_{1})\,\nu'(0)].
\]
By using (\ref{eq: D2g.v'}) and the fact that $y'(0)=-\nu_{1}^{\perp}$,
we get 
\[
(1-\kappa_{K}(y)d_{K}(x))\,\nu_{1}^{\perp}=0.
\]
Which is a contradiction.

We assumed that $1-\kappa_{K}(y)d_{K}(x)\neq0$, where $\kappa_{K}(y)$
is the $K$-curvature of $S_{1}$ at $y$. Let us also assume that
$B$ is small enough so that for $z\in B$ we have $1-\kappa_{K}(y(z))d_{K}(z)\neq0$.
Then, since $R_{K,0}\cap B=\emptyset$, we can repeat the proof of
Theorem \ref{thm: d_K is C^2} to deduce that $d_{K}$ is at least
$C^{2}$ on $B$. We also have $Dd_{K}(z)=\frac{\nu(y(z))}{\gamma^{\circ}(\nu(y(z)))}$
for $z\in B$. 

Next, let us show that if $B$ is small enough, the points on the
segment $l\cap\partial B$ have $y$ as the only $d_{K}$-closest
point on $\partial U$. This is obvious for points in $]x,y[$ by
Lemma \ref{lem: segment to the closest pt}; so we only need to consider
points $z$ on $l\cap\partial B$ such that $x\in]z,y[$. Take a sequence
$z_{i}\in B$ that converges to $z$. Then we can find points $x_{i}\in]z_{i},y(z_{i})[$
such that $x_{i}\to x$. Since we have $y(z_{i})=y(x_{i})\to y$,
$y$ is one of the $d_{K}$-closest points on $\partial U$ to $z$
by Lemma \ref{lem: cont of y}. Thus $y$ is the only $d_{K}$-closest
point on $\partial U$ to points in $]z,y[$; and we can make $B$
small enough to have the aforementioned property. We also make $B$
small enough so that $1-\kappa_{K}d_{K}\ne0$ on $\bar{B}$.

Now we claim that $Dd_{K}$ is uniformly continuos on $B$. Thus it
admits continuous extension to $\bar{B}$. It is enough to show that
$Dd_{K}(z)$ has a limit as $z$ approaches $\partial B$. Since we
can make $B$ smaller, we only need to consider $l\cap\partial B$.
Suppose $z_{i}\in B$ converge to $z$ on $l\cap\partial B$. Then
$y(z_{i})\to y$ and 
\[
Dd_{K}(z_{i})\to\frac{\nu_{1}}{\gamma^{\circ}(\nu_{1})}.
\]
Also note that $d_{K}$ is a linear function on $l\cap\partial B$,
and its derivative along $l$ is precisely the projection of $\frac{\nu_{1}}{\gamma^{\circ}(\nu_{1})}$
onto $l$. Therefore, $d_{K}$ is $C^{1}$ on $\bar{B}$.

Let $z\in l\cap\partial B$. Then $Dd_{K}(z)=\frac{\nu_{1}}{\gamma^{\circ}(\nu_{1})}$
from the side of $B$. Let us compute $Dd_{K}(z)$ from the other
side of $l$. We know that on the other side of $l$, $d_{K}(\cdot)=\gamma(\cdot-y)$.
Hence $Dd_{K}(z)=D\gamma(z-y)$. Now we have $z-y=d_{K}(z)D\gamma^{\circ}(\nu_{1})$
by Lemma (\ref{eq: parametrize by d_K}); so by (\ref{eq: homog}),
(\ref{eq: Dg(Dg0)}) we get 
\[
Dd_{K}(z)=D\gamma(d_{K}(z)D\gamma^{\circ}(\nu_{1}))=D\gamma(D\gamma^{\circ}(\nu_{1}))=\frac{\nu_{1}}{\gamma^{\circ}(\nu_{1})}.
\]
Therefore $Dd_{K}$ is continuous on $l\cap\partial B$ from both
sides, and thence $d_{K}$ is $C^{1}$ around $x$.

As $d_{K}$ is $C^{2}$ on both sides of $l\cap\partial B$, to show
that it is $C^{1,1}$ around $x$, it is enough to show that $D^{2}d_{K}$
remains bounded as we approach $l\cap\partial B$ from either side.
This is obvious on the side of $l$ where $d_{K}(\cdot)=\gamma(\cdot-y)$.
Let us consider the side where $B$ lies. It suffices to show that
\[
\textrm{tr}[(D^{2}d_{K})^{2}]=(D_{11}^{2}d_{K})^{2}+(D_{22}^{2}d_{K})^{2}+2(D_{12}^{2}d_{K})^{2}
\]
has limit as we approach $l\cap\partial B$. As shown in the proof
of Theorem \ref{thm: d_K is C^2}, the matrix of $D^{2}d_{K}$ in
the standard basis is similar to the matrix 
\[
\begin{bmatrix}0 & 0\\
0 & \Delta d_{K}
\end{bmatrix}.
\]
Since, the trace of similar matrices are the same, we get 
\[
\textrm{tr}[(D^{2}d_{K})^{2}]=(\Delta d_{K})^{2}.
\]
Now if $z_{i}\in B$ approach $l\cap\partial B$, then $y(z_{i})\to y$
and $\nu(y(z_{i}))\to\nu_{1}$. Thus, as $1-\kappa_{K}d_{K}\ne0$
on $\bar{B}$, $\Delta d_{K}(z_{i})$ has a limit by (\ref{eq: laplace d_K}).

To see that $d_{K}$ is not $C^{2}$ around $x$ in general, we can
compute $\Delta d_{K}$ from both sides of $l$, and see that in simple
examples they do not agree on $l$. For example, when $K$ is the
unit disk around the origin and $S_{1}$ is a line segment, we see
this phenomenon.

When $y$ is a non-strict reentrant corner, the argument is similar
to the above.
\end{proof}
\begin{thm}
Suppose $K,U$ satisfy the same assumptions as in Theorem \ref{thm: d_K is C^2, corners}.
Let $y\in S_{i}\subset\partial U$, and suppose that it is not a corner.
Also suppose that $S_{i}$ is a line segment if $\nu(y)=c\mu_{j}$
for some $c>0$. Then for some $r>0$ we have 
\[
B_{r}(y)\cap R_{K}=\emptyset.
\]
Furthermore, $d_{K}$ is at least $C^{2}$ up to $\partial U\cap B_{r}(y)$.
\end{thm}
\begin{proof}
First we claim that for some $r>0$ we have 
\[
B_{r}(y)\cap R_{K,0}=\emptyset.
\]
This is easy to show when $S_{i}$ is a line segment. Hence we assume
that $\nu(y)$, and consequently $\nu$ around $y$, is not a positive
multiple of any of $\mu_{j}$'s. Since $\partial U$ is at least $C^{2}$
around $y$, we can inscribe circles in $U$ which are tangent to
$\partial U$ and touch it only at one point near $y$. We can also
assume that the radii of these circles have a positive lower bound.
Now we can inscribe sets of the form $x-rK$ in each of these circles
so that it touches $\partial U$ at the same point that the circle
does. We can also assume that these $r$'s have a positive lower bound.
The reason is that $\partial K$ has positive curvature except at
a finite number of points, and those points are excluded by our assumption.
The existence of such inscribed sets implies the claim easily. Note
that as a consequence, $y$ is the $d_{K}$-closest point on $\partial U$
to some points in $U$.

Another way to prove this claim, is to assume the existence of a sequence
$x_{i}\in R_{K,0}$ that converges to $y$, and arrive at a contradiction
as we did in the proof of Theorem \ref{thm: d_K is C^2, corners}.

Now note that $\kappa_{K}$ is continuous, and hence bounded, on $\partial U$
around $y$. Thus, as $y(x)\to y$ when $x\to y$ by Lemma \ref{lem: cont of y},
we can make $r$ small enough so that $\kappa_{K}(y(x))d_{K}(x)\ne1$
for $x\in U\cap B_{r}(y)$. Therefore we have $B_{r}(y)\cap R_{K}=\emptyset$.

Next we show that $d_{K}$ is at least $C^{2}$ up to $\partial U\cap B_{r}(y)$.
To prove this, it is enough to show that $Dd_{K},D^{2}d_{K}$ have
limits as we approach $y$. Take $x\in U\cap B_{r}(y)$. Then $Dd_{K}(x)=\frac{\nu(y(x))}{\gamma^{\circ}(\nu(y(x)))}$.
When $x\to y$ we have $y(x)\to y$, so by continuity of $\nu$ we
get $Dd_{K}(x)\to\frac{\nu(y)}{\gamma^{\circ}(\nu(y))}$ as desired.

To show the same for $D^{2}d_{K}$, we use (\ref{eq: laplace d_K})
and the continuity of $\kappa,\kappa_{K}$ on $\partial U\cap B_{r}(y)$,
to get 
\[
\Delta d_{K}(x)\to\frac{-\kappa(y)|\nu(y)|^{3}|D\gamma^{\circ}(\nu(y))|^{2}}{\gamma^{\circ}(\nu(y))^{3}}.
\]
On the other hand we have 
\[
\zeta(x):=\Big(\frac{x-y(x)}{|x-y(x)|}\Big)^{\perp}=\Big(\frac{D\gamma^{\circ}(\nu(y(x)))}{|D\gamma^{\circ}(\nu(y(x)))|}\Big)^{\perp}\to\Big(\frac{D\gamma^{\circ}(\nu(y))}{|D\gamma^{\circ}(\nu(y))|}\Big)^{\perp}.
\]
Note that here we used (\ref{eq: K-normal}). Hence again by (\ref{eq: laplace d_K})
we see that $D^{2}d_{K}(x)$ has a limit as $x\to y$.
\end{proof}
\begin{rem}
When $\nu(y)$ is a positive multiple of a $\mu_{j}$ and the curvature
of $S_{i}$ is positive at $y$, $R_{K,0}$ can have $y$ as a limit
point. The same thing happens when $y$ is a nonreentrant corner.
When $y$ is a strict reentrant corner and we approach it from the
region between its inward $K$-normals, $Dd_{K}$ will not have a
limit and $D^{2}d_{K}$ will blow up, by (\ref{eq: homog}).

Finally, when $y$ is a non-strict reentrant corner, with a slight
modification of the above proof we can show that $R_{K}$ has a positive
distance from $y$, and $d_{K}$ is $C^{1}$ up to $y$.
\end{rem}

\section{\label{sec: Characterizing ridge}Characterizing the ridge}

At this point we have the tools to specify the points in the $d_{K}$-ridge
of $U$.

\begin{thm}
\label{thm: ridge}Suppose $K,U$ satisfy the same assumptions as
in Theorem \ref{thm: d_K is C^2, corners}. Then the $d_{K}$-ridge
consists of $R_{K,0}$ and those points $x$ outside of it at which
\[
\kappa_{K}(y(x))d_{K}(x)=1.
\]
Here, if $y=y(x)$ is a reentrant corner, then $x-y$ must be parallel
to one of the inward $K$-normals at $y$, and $\kappa_{K}$ is the
$K$-curvature of the corresponding boundary part.
\end{thm}
\begin{proof}
So far, we showed that $R_{K}$ contains $R_{K,0}$. We also showed
in Theorems \ref{thm: d_K is C^2}, \ref{thm: d_K is C^2, corners}
that every point outside $R_{K,0}$ which is not described in the
statement of the theorem is not in $R_{K}$, i.e. those points at
which $1-\kappa_{K}d_{K}\ne0$, and those points between the $K$-normals
of a strict reentrant corner which have that corner as the $d_{K}$-closest
point.

Now to prove theorem's assertion, first suppose that $y\in S_{1}\cap S_{2}$
is a reentrant corner and $1-\kappa_{K}(y)d_{K}(x)=0$, where $\kappa_{K}$
is the $K$-curvature of $S_{1}$. Then $\kappa_{K}(y)=\frac{1}{d_{K}(x)}>0$,
and consequently $\kappa(y)>0$, where $\kappa$ is the ordinary curvature
of $S_{1}$. Consider the line segment $]x,y[$. On this segment,
$y$ is the unique $d_{K}$-closest point on $\partial U$; so $d_{K}$
decreases linearly as we move from $x$ to $y$. Hence $1-\kappa_{K}d_{K}>0$
on $]x,y[$. Thus, as seen in the proof of Theorem \ref{thm: d_K is C^2, corners},
$d_{K}$ is at least $C^{2}$ on an open set $B$, which is on one
side of $]x,y[$ and has $]x,y[$ as part of its boundary. Also, the
$d_{K}$-closest points to points of $B$ lie on $S_{1}$. Choose
a sequence $z_{i}\in B$ that converges to $x$. Then $y(z_{i})\to y$
by Lemma \ref{lem: cont of y}; and by continuity of $\kappa_{K},\kappa$
on $S_{1}$ we have 
\begin{eqnarray*}
\kappa_{K}(y(z_{i}))\to\kappa_{K}(y), &  & \kappa(y(z_{i}))\to\kappa(y).
\end{eqnarray*}
Thus in particular, $\kappa(y(z_{i}))>0$ for $i$ large enough. Since
on $B$, $\Delta d_{K}$ is given by (\ref{eq: laplace d_K}), $\Delta d_{K}(z_{i})$
blows up as $z_{i}\to x$. Therefore, $d_{K}$ can not be $C^{1,1}$
in any neighborhood of $x$.

If $y$ is not a reentrant corner, we can repeat the above argument
by simply approaching $x$ through points of $]x,y[$.
\end{proof}
The proof of the following theorem is a variant of the proof of a
similar result in \citep{MR2336304}.
\begin{thm}
\label{thm: 1-kd>0}Suppose $K,U$ satisfy the same assumptions as
in Theorem \ref{thm: d_K is C^2, corners}. Then for $x\in U-R_{K}$
we have 
\[
1-\kappa_{K}(y(x))d_{K}(x)>0.
\]
Here, if $y=y(x)$ is a reentrant corner, then $x-y$ must be parallel
to one of the inward $K$-normals at $y$, and $\kappa_{K}$ is the
$K$-curvature of the corresponding boundary part.
\end{thm}
\begin{proof}
We will show that $1-\kappa_{K}(y)d_{K}(x)\ge0$. This gives the desired
result, since we know that $1-\kappa_{K}(y)d_{K}(x)\ne0$. If $\kappa_{K}(y)=0$
the relation holds trivially, so suppose it is nonzero. Note that
as shown in the proof of Theorem \ref{thm: d_K is C^2}, $y$ cannot
be a nonreentrant corner, and the inward unit normal to $\partial U$
at $y$ is not equal to any of the $\mu_{i}$'s, since we assumed
that $\kappa_{K}(y)\ne0$. 

Let $t\mapsto y(t)$ be a smooth nondegenerate parametrization of
a segment of $\partial U$ around $y$ which has $y$ as an endpoint,
and $y(0)=y$. We assume that the direction of the parametrization
is such that $\nu:=(y')^{\perp}$ is an inward normal to $\partial U$.
Consider the function $t\mapsto\gamma(x-y(t))$. It has a minimum
at $t=0$; and there, its first derivative is 
\[
\langle D\gamma(x-y),-y'(0)\rangle=\langle D\gamma(x-y),\nu^{\perp}\rangle.
\]
But by (\ref{eq: parametrize by d_K}) we have $x-y=d_{K}(x)D\gamma^{\circ}(\nu)$.
Hence by (\ref{eq: homog}), (\ref{eq: Dg(Dg0)}), the first derivative
vanishes at $t=0$. Thus the second derivative must be nonnegative
at $t=0$, i.e. 
\[
\langle D^{2}\gamma(x-y)\,y'(0),y'(0)\rangle-\langle D\gamma(x-y),y''(0)\rangle\ge0.
\]
By using homogeneity of $D\gamma,D^{2}\gamma$, and (\ref{eq: Dg(Dg0)}),
(\ref{eq: parametrize by d_K}) we get 
\begin{equation}
\frac{1}{d_{K}(x)}\langle D^{2}\gamma(D\gamma^{\circ}(\nu))\,\nu^{\perp},\nu^{\perp}\rangle+\langle\frac{\nu}{\gamma^{\circ}(\nu)},(\nu')^{\perp}\rangle\ge0.\label{eq: 1 in thm 1-kd>0}
\end{equation}

On the other hand, by differentiating (\ref{eq: Dg(Dg0)}) we get
\[
\underset{k}{\sum}D_{ik}^{2}\gamma(D\gamma^{\circ}(\nu))D_{kj}^{2}\gamma^{\circ}(\nu)=\frac{1}{\gamma^{\circ}(\nu)}\delta_{ij}-\frac{\nu_{i}D_{j}\gamma^{\circ}(\nu)}{\gamma^{\circ}(\nu)^{2}}.
\]
Multiplying both sides by $\nu_{i}^{\perp},\nu'_{j}$ and summing
over $i,j$ gives us 
\begin{eqnarray*}
 & \langle D^{2}\gamma(D\gamma^{\circ}(\nu))\,\nu^{\perp},D^{2}\gamma^{\circ}(\nu)\,\nu'\rangle & =\underset{i,j,k}{\sum}\nu_{i}^{\perp}D_{ik}^{2}\gamma(D\gamma^{\circ}(\nu))D_{kj}^{2}\gamma^{\circ}(\nu)\nu'_{j}\\
 &  & =\frac{1}{\gamma^{\circ}(\nu)}\langle\nu^{\perp},\nu'\rangle.
\end{eqnarray*}
And by (\ref{eq: D2g.v'}) we obtain 
\[
\langle D^{2}\gamma(D\gamma^{\circ}(\nu))\,\nu^{\perp},\nu^{\perp}\rangle=\frac{1}{\kappa_{K}(y)\gamma^{\circ}(\nu)}\langle\nu^{\perp},\nu'\rangle.
\]

If we insert this in (\ref{eq: 1 in thm 1-kd>0}) and use the fact
that the ordinary curvature is given by $\kappa=\frac{\langle\nu^{\perp},\nu'\rangle}{|\nu|^{3}}$,
we deduce that 
\[
0\le\frac{1}{\kappa_{K}(y)d_{K}(x)\gamma^{\circ}(\nu)}\langle\nu^{\perp},\nu'\rangle+\frac{1}{\gamma^{\circ}(\nu)}\langle\nu,(\nu')^{\perp}\rangle=\frac{|\nu|^{3}\kappa(y)[1-\kappa_{K}(y)d_{K}(x)]}{\kappa_{K}(y)d_{K}(x)\gamma^{\circ}(\nu)}.
\]
Now, as $\kappa,\kappa_{K}$ have the same sign, we must have $1-\kappa_{K}d_{K}\ge0$
as desired.
\end{proof}
\begin{rem}
\label{rem: nonstrict corner}Suppose $y$ is a non-strict reentrant
corner, and $\kappa_{K,1}>\kappa_{K,2}$, where $\kappa_{K,1},\kappa_{K,2}$
are the $K$-curvatures at $y$ from different sides. Then by Theorem
\ref{thm: ridge}, if 
\[
x_{i}:=y+\frac{1}{\kappa_{K,i}}\nu_{K}(y)
\]
have $y$ as the only $d_{K}$-closest point on $\partial U$, then
they belong to the $d_{K}$-ridge. But, Theorem \ref{thm: 1-kd>0}
implies that $y$ cannot be a $d_{K}$-closest point to any point
on the segment $]x_{1},x_{2}]$. Thus, $x_{1}$ is the only point
along the $K$-normal at $y$, that can belong to $R_{K}-R_{K,0}$
and have $y$ as the unique $d_{K}$-closest point on $\partial U$.
\end{rem}

\bibliographystyle{plainnat}
\bibliography{Newest-Bibliography}

\end{document}